\documentclass[11pt]{amsart}
\usepackage{amssymb,amsmath,amsthm,newlfont,enumerate}
\theoremstyle{plain}
\newtheorem{theorem}{Theorem}[section]
\newtheorem{corollary}[theorem]{Corollary}
\newtheorem{lemma}[theorem]{Lemma}

\theoremstyle{definition}
\newtheorem{definition}[theorem]{Definition}

\theoremstyle{remark}

\newcommand{\DD}{{\mathbb D}}

\newcommand{\TT}{{\mathbb T}}

\newcommand{\cH}{{\cal H}}

\let\Re\undefined
\DeclareMathOperator{\Re}{Re}
\renewcommand{\hat}{\widehat}

\begin{document}

\title{Constructive approximation in de Branges--Rovnyak spaces}
\author{O. El-Fallah}
\address{Laboratoire Analyse et Applications (URAC/03),
Universit\'e Mohamed V, B.P. 1014 Rabat, Morocco}
\email{elfallah@fsr.ac.ma} 
\thanks{Research of OE supported by  CNRST (URAC/03) and  Acad\'emie Hassan II des sciences et techniques}

\author{E. Fricain}
\address{Laboratoire Paul Painlev\'e, UFR des Math\'ematiques,
Universit\'e des Sciences et Technologies Lille~1, 59655 Villeneuve
d'Ascq Cedex, France.}
\email{emmanuel.fricain@math.univ-lille1.fr} 
\thanks{Research of EF supported by CEMPI}

\author{K. Kellay}
\address{Institut de Math\'ematiques de Bordeaux, Universit\'e de Bordeaux, 351 cours de la Lib\'eration,
33405 Talence Cedex, France}
\email{karim.kellay@math.u-bordeaux1.fr} 
\thanks{Research of KK supported by UMI-CRM}

\author{J. Mashreghi}
\address{D\'epartement de math\'ematiques et de statistique, 
Universit\'e Laval, 1045 avenue de la M\'edecine, Qu\'ebec (QC), G1V 0A6, Canada}
\email{javad.mashreghi@mat.ulaval.ca} 
\thanks{Research of JM supported by NSERC}

\author{T. Ransford}
\address{D\'epartement de math\'ematiques et de statistique, 
Universit\'e Laval, 1045 avenue de la M\'edecine, Qu\'ebec (QC), G1V 0A6, Canada}
\email{ransford@mat.ulaval.ca}
\thanks{Research of TR supported by NSERC and the Canada research chairs program}

\begin{abstract}
In most classical holomorphic function spaces on the unit disk, a function $f$ can be approximated in the norm of the space by its dilates $f_r(z):=f(rz)~(r<1)$.
We show that this is \emph{not} the case for the de Branges--Rovnyak spaces $\cH(b)$. More precisely, we give an example of a non-extreme point $b$ of the unit ball of $H^\infty$ and a function $f\in\cH(b)$ such that $\lim_{r\to1^-}\|f_r\|_{\cH(b)}=\infty$.

It is  known that, if $b$ is a non-extreme point of the unit ball of $H^\infty$, then polynomials are dense in $\cH(b)$. We give the first constructive proof of this fact.
\end{abstract}

\date{6 January  2015}

\maketitle

\section{Introduction}\label{S:intro}

The de Branges--Rovnyak spaces  
are a family of subspaces $\cH(b)$ of the Hardy space $H^2$,
para\-metrized by elements $b$ of the closed unit ball of $H^\infty$.
We shall give the precise definition  in \S\ref{S:background}.
In general $\cH(b)$ is not closed in $H^2$, 
but it carries its own norm $\|\cdot\|_{\cH(b)}$ making it a Hilbert space.

The spaces $\cH(b)$ were introduced by de Branges and Rovnyak 
in the appendix of \cite{dBR66a} and further studied in their book \cite{dBR66b}.
The initial motivation was to provide canonical model spaces
for certain types of contractions on Hilbert spaces.
Subsequently it was realized that 
these spaces  have several interesting connections
with other topics in complex analysis and operator theory.
For background information
we refer to the books of 
de Branges and Rovnyak \cite{dBR66b}, Sarason \cite{Sa94},
and the forthcoming monograph of 
Fricain and Mashreghi \cite{FM12}.

The general theory of $\cH(b)$-spaces  splits  into two cases, according to whether $b$ is an
extreme point or a non-extreme point of the unit ball of $H^\infty$. For example,
if $b$ is non-extreme, then $\cH(b)$ contains all functions holomorphic in a neighborhood of the closed unit disk $\overline{\DD}$,
whereas  if $b$ is extreme, then $\cH(b)$ contains very few such functions. In particular, $\cH(b)$ contains the polynomials if and only if $b$ is non-extreme, and in this case,
the polynomials are  dense in $\cH(b)$. Proofs of all these facts can be found in Sarason's book \cite{Sa94}.

The  density of polynomials is proved in \cite{Sa94} by showing that their orthogonal complement in $\cH(b)$ is zero.
The proof is non-constructive in the sense that it gives no clue how to find polynomial approximants to a given function. 
We  know of no published work describing  constructive methods of polynomial approximation in $\cH(b)$, 
and it is surely of interest to have such methods available. 

Perhaps the most natural approach is to try using dilations.
Writing $f_r(z):=f(rz)$, the idea is to approximate $f$ by $f_r$ for  some $r<1$, 
and then $f_r$ 
by the partial sums of its Taylor series.
This idea works in many function spaces, but, as we shall see, it fails dismally in $\cH(b)$, at least for certain choices of $b$. Indeed, it can happen that
$\lim_{r\to1^-}\|f_r\|_{\cH(b)}=\infty$, even though $f\in\cH(b)$. We shall prove this in \S\ref{S:dilation}, thereby answering a question posed in \cite{CGR10}.

This phenomenon has other negative consequences, among them the surprising fact that the formula for $\|f\|_{\cH(b)}$ in terms of the Taylor coefficients of $f$, previously known to hold for $f$ holomorphic on a neighborhood of $\overline{\DD}$, actually breaks down for general $f\in\cH(b)$.
It also shows that, in general, neither the Taylor partial sums of $f$, nor their Ces\`aro means need converge to $f$ in $\cH(b)$. 

So, to construct polynomial approximants to functions in $\cH(b)$, a different idea is needed. In \S\ref{S:polyapprox} we shall describe a scheme that achieves this based on Toeplitz operators. The method presupposes that $b$ is non-extreme (as it must), but one of its consequences is an approximation theorem for Toeplitz operators that extends even to the case when $b$ is extreme. We shall give a proof of this extension in \S\ref{S:Toeplitz}.

\section{Background on $\cH(b)$-spaces}\label{S:background}

Given $\psi\in L^\infty(\TT)$, the corresponding Toeplitz operator $T_\psi:H^2\to H^2$ is defined by
\[
T_\psi f:=P_+(\psi f) \qquad(f\in H^2),
\]
where $P_+:L^2(\TT)\to H^2$ denotes the othogonal projection of $L^2(\TT)$ onto $H^2$. 
Clearly $T_\psi$ is a bounded operator on $H^2$ with $\|T_\psi\|\le\|\psi\|_{L^\infty(\TT)}$ (in fact, by a theorem of Brown and Halmos, $\|T_\psi\|=\|\psi\|_{L^\infty(\TT)}$, but we do not need this).
If $h\in H^\infty$, then $T_h$ is simply the operator of multiplication by $h$ and its adjoint is $T_{\overline{h}}$.
Consequently, if $h,k\in H^\infty$, then
$T_{\overline{h}}T_{\overline{k}}=T_{\overline{hk}}=T_{\overline{k}}T_{\overline{h}}$, a useful fact that we shall exploit frequently in what follows.

\begin{definition}
Let $b\in H^\infty$
with $\|b\|_{H^\infty}\le1$.
The associated \emph{de Branges--Rovnyak space} 
$\cH(b)$ is the image of $H^2$ under the operator $(I-T_ bT_{\overline b})^{1/2}$.
We define a norm  on $\cH(b)$ making $(I-T_bT_{\overline{b}})^{1/2}$ a partial isometry
from $H^2$ onto $\cH(b)$, namely
\[
\|(I-T_bT_{\overline{b}})^{1/2}f\|_{\cH(b)}:=\|f\|_{H^2}
\qquad(f\in H^2\ominus \ker(I-T_bT_{\overline{b}})^{1/2})).
\]
\end{definition}

This is the definition of $\cH(b)$ as given in \cite{Sa94}. The original definition of de Branges and Rovnyak,
based on the notion of complementary space, is different but equivalent. An explanation of the equivalence can be found in \cite[pp.7--8]{Sa94}. 
A third approach is to start from the positive kernel
\[
k_w^b(z):=\frac{1-\overline{b(w)}b(z)}{1-\overline{w}z}
\qquad(z,w\in\DD),
\]
and to define $\cH(b)$ as the reproducing kernel Hilbert space associated with this kernel.

As mentioned in the introduction, the theory of $\cH(b)$-spaces is pervaded by a fundamental dichotomy,
namely whether $b$ is or is not an extreme  point of the unit ball of $H^\infty$. 
This dichotomy is illustrated by following result.

\begin{theorem}\label{T:nonextreme}
Let $b\in H^\infty$ with $\|b\|_{H^\infty}\le 1$. The following statements are equivalent:
\begin{itemize}
\item[\rm(i)] $b$ is a non-extreme point of the unit ball of $H^\infty$;
\item[\rm(ii)] $\log (1-|b|^2)\in L^1(\TT)$;
\item[\rm(iii)]  $\cH(b)$ contains all functions holomorphic in a neighborhood of $\overline{\DD}$;
\item[\rm(iv)] $\cH(b)$ contains all polynomials.
\end{itemize}
\end{theorem}

\begin{proof}
The equivalence between (i) and (ii) is proved in \cite[Theorem~7.9]{Du00}. 
That (i) implies (iii) is proved in \cite[\S IV-6]{Sa94}.
Clearly (iii) implies (iv).
That (iv) implies (i) follows from \cite[\S V-1]{Sa94}.
\end{proof}

Henceforth we shall simply say that $b$ is `extreme' or `non-extreme', it being understood that this relative to the unit ball of $H^\infty$.

From the equivalence between (i) and (ii), it follows that, if  $b$ is non-extreme,
then there is an  outer function $a$ such that $a(0)>0$ and  $|a|^2+|b|^2=1$ a.e.\ on $\TT$ (see \cite[\S IV-1]{Sa94}). 
The function $a$ is uniquely determined by $b$. We shall call $(b,a)$ a \emph{pair}. 
The following result gives a useful characterization of $\cH(b)$ in this case.

\begin{theorem}[\protect{\cite[\S IV-1]{Sa94}}]\label{T:f+}
Let $b$ be non-extreme, let  $(b,a)$ be a pair and let $f\in H^2$.
Then $f\in\cH(b)$ if and only if 
$T_{\overline b}f\in T_{\overline a}(H^2)$.
In this case, there exists a unique function $f^+\in H^2$
such that $T_{\overline b}f=T_{\overline a}f^+$, and
\begin{equation}\label{E:f+}
\|f\|_{\cH(b)}^2=\|f\|_{H^2}^2+\|f^+\|_{H^2}^2.
\end{equation}
\end{theorem}

We end this section with an example that was studied in \cite{Sa97}. Let
\begin{equation}\label{E:b0}
b_0(z):=\frac{\tau z}{1-\tau^2 z},
\end{equation}
where $\tau:=(\sqrt{5}-1)/2$. 
The equivalence between (i) and (ii) in Theorem~\ref{T:nonextreme} shows that $b_0$ is non-extreme,
and a calculation shows that the function $a_0$ making
$(b_0,a_0)$ a pair is given by
\[
a_0(z)=\frac{\tau(1-z)}{1-\tau^2z}.
\]
It was shown in \cite{Sa97} that $b_0$ has the special property that
$\|f_r\|_{\cH(b_0)}\le\|f\|_{\cH(b_0)}$ for all $f\in\cH(b_0)$ and all $r<1$. 
Using a standard argument of reproducing kernel Hilbert spaces, it is easy to see that this implies that 
$\lim_{r\to1^-}\|f_r-f\|_{\cH(b_0)}=0$.
As we shall see in the next section,
this property is not shared by general $b$.

\section{Dilation in $\cH(b)$}\label{S:dilation}

Our principal goal in this section is to prove the following theorem.

\begin{theorem}\label{T:negative}
Let  $b:=b_0B^2$, where $b_0$ is the function given by \eqref{E:b0}, and $B$ is the Blaschke product
with zeros at $w_n:=1-4^{-n}~(n\ge1)$. 
Let $f(z):=\sum_{n\ge1}2^{-n}/(1- w_nz)$.
Then $b$ is non-extreme,  $f\in\cH(b)$, and we have 
\begin{equation}\label{E:negative}
\lim_{r\to 1^-}|(f_r)^+(0)|=\infty
\quad\text{and}\quad
\lim_{r\to1^-}\|f_r\|_{\cH(b)}=\infty.
\end{equation}
\end{theorem}

Notice that, by Theorem~\ref{T:f+}, if $b$ is non-extreme and $f\in\cH(b)$, then
\[
\|f\|_{\cH(b)}\ge \|f^+\|_{H^2}\ge|f^+(0)|.
\]
Thus the second conclusion in \eqref{E:negative} is actually a consequence of the first. We shall therefore concentrate our attention on the first conclusion.

To simplify the notation in what follows, we shall write
$k_w(z):=1/(1-\overline{w}z)$, the Cauchy kernel.
It is the reproducing kernel for $H^2$ in the sense that
$f(w)=\langle f,k_w\rangle_{H^2}$ for all $f\in H^2$ and $w\in\DD$. In particular,
$\|k_w\|_{H^2}^2=\langle k_w,k_w\rangle_{H^2}=k_w(w)=1/(1-|w|^2)$.
We remark that $k_w$ has the useful property that
$T_{\overline{h}}(k_w)=\overline{h(w)}k_w$ for all $h\in H^\infty$. Indeed, given $g\in H^2$, we have
\[
\langle g, T_{\overline{h}}(k_w)\rangle_{H^2}
=\langle hg,k_w,\rangle_{H^2}
=h(w)g(w)
=h(w)\langle g,k_w\rangle_{H^2}
=\langle g,\overline{h(w)}k_w,\rangle_{H^2}.
\]

The proof of Theorem~\ref{T:negative} depends on two lemmas.
The first lemma provides a class of functions $f$ for which $(f_r)^+(0)$ is readily identifiable.

\begin{lemma}\label{L:sumcnkwn}
Let $b$ be  non-extreme, let $(b,a)$ be a pair and let $\phi:=b/a$. Let
\[
f:=\sum_{n\ge1} c_n k_{w_n},
\]
where $(w_n)_{n\ge1}$ are zeros of $b$ and  $(c_n)_{n\ge1}$ are complex numbers with 
$\sum_n|c_n|(1-|w_n|)^{-1/2}<\infty$.
Then $f\in\cH(b)$ and
\begin{equation}\label{E:sum}
(f_r)^+(0)=\sum_{n\ge1}c_n\overline{\phi(rw_n)}
\qquad(0<r<1).
\end{equation}
\end{lemma}

\begin{proof}
The series defining $f$ clearly converges absolutely in $H^2$.
Also, since 
$T_{\overline{b}}k_{w_n}=\overline{b(w_n)}k_{w_n}=0$ for all $n$, we have $T_{\overline{b}}f=0$, and consequently $f\in \cH(b)$ by Theorem~\ref{T:f+}.

Now fix $r\in(0,1)$ and consider
\[
g:=\sum_{n\ge1}c_n\overline{\phi(rw_n)}k_{w_n}.
\]
As $(\phi(rw_n))_{n\ge1}$ is a bounded sequence, this series also converges absolutely in $H^2$, and a 
simple calculation gives $T_{\overline{b}}(f_r)=T_{\overline{a}}(g_r)$. Thus $f_r\in\cH(b)$ and $(f_r)^+=g_r$. In particular \eqref{E:sum} holds.
\end{proof}

The second lemma is a technical result about Blaschke products.

\begin{lemma}\label{L:bp}
Let $B$ be an infinite Blaschke product whose zeros $(w_n)_{n\ge1}$ lie in $(0,1)$ and satisfy
\begin{equation}\label{E:bp}
0<\alpha\le\frac{1-w_{n+1}}{1-w_n}\le\beta <\frac{1}{2} \qquad(n\ge1).
\end{equation}
Then there exists a constant $C>0$ such that 
\[
|B(rw_n)|\ge C \qquad(w_n\le r\le w_{n+1},~n\ge1).
\]
\end{lemma}

\begin{proof}
We have
\[
|B(rw_n)|=\prod_{k=1}^\infty \rho(w_k,rw_n),
\]
where $\rho$ denotes  the pseudo-hyperbolic metric on $\DD$, namely $\rho(z,w):=|z-w|/|1-\overline{w}z|$. 
The condition \eqref{E:bp} implies that 
$w_{n-1}<w_n$ for all $n$, and even that
$w_{n-1}<w_n^2$. Indeed, we have
\begin{equation}\label{E:beta}
1-w_n^2\le 2(1-w_n)\le 2\beta(1-w_{n-1})<1-w_{n-1}.
\end{equation}
It follows that, if $r\in[w_n,w_{n+1}]$, then
$rw_n\in[w_n^2,w_nw_{n+1}]\subset(w_{n-1},w_n)$, and 
consequently
\begin{equation}\label{E:prod4}
|B(rw_n)|\ge\Bigl(\prod_{k=1}^{n-2}\rho(w_k,w_{n-1})\Bigr)\times\rho(w_{n-1},w_n^2)\times\rho(w_n,w_nw_{n+1})\times\Bigl(\prod_{k=n+1}^\infty\rho(w_k,w_n)\Bigr).
\end{equation}
Thus the lemma will be proved if we can show that each of the four terms on the right-hand side of \eqref{E:prod4}
 is bounded below by a positive constant independent of $n$.

By \cite[Theorem~9.2]{Du00}, the condition \eqref{E:bp} implies that the sequence $(w_n)$ is uniformly separated, in other words, there exists a constant $C'>0$ such that
\[
\prod_{k\ne j}\rho(w_k,w_j)\ge C' \qquad(j\ge1).
\]
Applying this with $j=n-1$ and $j=n$ takes care of the first and fourth terms in \eqref{E:prod4}. 

For the second term in \eqref{E:prod4}, note that  \eqref{E:beta} gives
$w_n^2-w_{n-1}\ge (1-2\beta)(1-w_{n-1})$, and clearly also
$1-w_n^2w_{n-1}\le 1-w_{n-1}^2\le 2(1-w_{n-1})$,
whence
\[
\rho(w_{n-1},w_n^2)
=\frac{w_n^2-w_{n-1}}{1-w_n^2w_{n-1}}
\ge \frac{1-2\beta}{2}.
\]

Finally, for the third term in \eqref{E:prod4}, we observe that
$w_n-w_nw_{n+1}\ge w_1(1-w_{n+1})$ and also
$1-w_n^2w_{n+1}=(1-w_n)+(w_n-w_n^2)+(w_n^2-w_n^2w_{n+1})
\le 3(1-w_n)$, whence
\[
\rho(w_n,w_nw_{n+1})
=\frac{w_n-w_nw_{n+1}}{1-w_n^2w_{n+1}}
\ge \frac{w_1}{3}\frac{1-w_{n+1}}{1-w_n}\ge \frac{w_1}{3}\alpha.\qedhere
\]
\end{proof}

\begin{proof}[Proof of Theorem~\ref{T:negative}]
As remarked in \S\ref{S:background},
the function $b_0$ is  non-extreme and the function $a_0$ making $(b_0,a_0)$ a pair satisfies
$\phi_0:=b_0/a_0=z/(1-z)$.
As $b$ and $b_0$ have the same outer factors, 
it follows that
$b$ is non-extreme and
the function $a$ making $(b,a)$ a pair is just $a_0$. 
Hence $\phi:=b/a=B^2b_0/a_0=B^2\phi_0$.

By Lemma~\ref{L:sumcnkwn}, we have $f\in\cH(b)$. The lemma also gives that
\[
(f_r)^+(0)=\sum_{n\ge1}2^{-n}\overline{\phi(rw_n)}=\sum_{n\ge1}2^{-n}B(rw_n)^2\frac{rw_n}{1-rw_n}.
\] 
As the terms in this series are non-negative,
each one of them provides a lower bound for the sum. Given $r\in [w_1,1)$, we choose $n$ so that
$w_n\le r\le w_{n+1}$. By Lemma~\ref{L:bp} we have $|B(rw_n)|\ge C>0$, where $C$ is a constant independent
of $r$ and $n$. Thus 
\[
(f_r)^+(0)\ge 2^{-n}C^2\frac{rw_n}{1-rw_n}\ge 2^{-n}C^2\frac{w_n^2}{1-w_n^2}\asymp 2^n\asymp (1-r)^{-1/2}.
\]
In particular $(f_r)^+(0)\to\infty$ as $r\to1^-$,
as claimed. Finally,
as already remarked, this implies that
$\|f_r\|_{\cH(b)}\to\infty$ as $r\to1^-$.
\end{proof}

We now present some consequences of this result.

\begin{corollary}
Let $b,f$ be as in Theorem~\ref{T:negative}.
Then the Taylor partial sums $s_n(f)$ of $f$ and their 
Ces\`aro means $\sigma_n(f)$ satisfy
\[
\limsup_{n\to\infty}\|s_n(f)\|_{\cH(b)}=\infty
\quad\text{and}\quad
\limsup_{n\to\infty}\|\sigma_n(f)\|_{\cH(b)}=\infty.
\]
\end{corollary}

\begin{proof}
This follows immediately from Theorem~\ref{T:negative}
and the elementary identities
\[
f_r=\sum_{n\ge0}(1-r)r^n s_n(f)
\quad\text{and}\quad
f_r=\sum_{n\ge0}(n+1)(1-r)^2r^n\sigma_n(f).\qedhere
\]
\end{proof}

Let $b$ be non-extreme, let $(b,a)$ be a pair and let $\phi:=b/a$, say $\phi(z)=\sum_{j\ge0}\hat{\phi}(j)z^j$. 
It was shown in \cite[Theorem~4.1]{CGR10}  that,
if $f$ is holomorphic in a neighborhood of $\overline{\DD}$,
say $f(z)=\sum_{k\ge0}\hat{f}(k)z^k$, then the series 
$\sum_{j\ge0}\hat{f}(j+k)\overline{\hat{\phi}(j)}$ converges absolutely  for each $k$, and 
\begin{equation}\label{E:hbnorm}
\|f\|_{\cH(b)}^2=\sum_{k\ge0}|\hat{f}(k)|^2
+\sum_{k\ge0}\Bigl|\sum_{j\ge0}\hat{f}(j+k)\overline{\hat{\phi}(j)}\Bigr|^2.
\end{equation}
It was left open whether the same formula holds for all $f\in\cH(b)$. Using Theorem~\ref{T:negative},
we can now show that it does not.

\begin{corollary}
Let $b,f$ be as in Theorem~\ref{T:negative}. 
Then $\sum_{j\ge0}\hat{f}(j)\overline{\hat{\phi}(j)}$ diverges.
\end{corollary}

\begin{proof}
For $r\in(0,1)$, the dilated function $f_r$ is holomorphic in a neighborhood of $\overline{\DD}$, and the argument in \cite{CGR10} that establishes the formula \eqref{E:hbnorm}
shows that $(f_r)^+(0)=\sum_{j\ge0}r^j\hat{f}(j)\overline{\hat{\phi}(j)}$. If $b,f$ are as in Theorem~\ref{T:negative} then $(f_r)^+(0)\to\infty$ as $r\to 1^-$,
in other words, $\lim_{r\to 1^-}\sum_{j\ge0}r^j\hat{f}(j)\overline{\hat{\phi}(j)}=\infty$. By Abel's theorem, it follows that the series
$\sum_{j\ge0}\hat{f}(j)\overline{\hat{\phi}(j)}$ diverges.
\end{proof}

In  Theorem~\ref{T:negative}, we chose $b_0$ so as to have a simple concrete example. With a more astute choice, we can prove more, obtaining examples where $\|f_r\|_{\cH(b)}$ grows `fast'. There is a limit on how fast it can grow:
it was shown in \cite[Theorem~5.2]{CGR10} that, if $b$ is non-extreme and $f\in\cH(b)$, then
$\log^+\|f_r\|_{\cH(b)}=o((1-r)^{-1})$
as $r\to 1^-$. 
We now prove that this estimate is sharp.

\begin{theorem}
Let $\gamma:(0,1)\to(1,\infty)$ be a function such that $\log\gamma(r)=o((1-r)^{-1})$.
Then there exist $b$ non-extreme and $f\in \cH(b)$ such that 
$\|f_r\|_{\cH(b)}\ge\gamma(r)$ for all $r$ in some interval $(r_0,1)$.
\end{theorem}

\begin{proof}

Let $\phi_1$ be any function in the Smirnov class $N^+$ that is positive and increasing on $(0,1)$. To say that $\phi_1\in N^+$ means we can write $\phi_1=b_1/a_1$, where $a_1,b_1\in H^\infty$ and $a_1$ is outer. Multiplying $a_1$ and $b_1$ by an appropriately chosen outer function, we may further ensure that $|a|^2+|b|^2=1$ a.e.\ on $\TT$ and that $a_1(0)>0$, in other words,
that $(b_1,a_1)$ is a pair.
Repeating the proof of Theorem~\ref{T:negative} with $b_0$ replaced by $b_1$ (but with the same $B$),
we obtain $f\in\cH(b)$ such that
\[
(f_r)^+(0)\ge C^2(1-r)^{1/2}\phi_1(8r-7)
\qquad(3/4<r<1).
\]
Since $\log\gamma(r)=o((1-r)^{-1})$, it is possible to choose $\phi_1$ so that right-hand side  exceeds $\gamma(r)$ for all $r$ sufficiently close to $1$. For these $r$, we therefore have $\|f_r\|_{\cH(b)}\ge (f_r)^+(0)\ge \gamma(r)$.
\end{proof}

\section{Polynomial approximation in $\cH(b)$}\label{S:polyapprox}

In this section we present a recipe for polynomial approximation in $\cH(b)$ when $b$ is non-extreme. It is based upon three lemmas.

\begin{lemma}\label{L:Toeplitzpoly}
Let $h\in H^\infty$ and let $p$ be a polynomial. Then $T_{\overline{h}}p$ is a polynomial.
\end{lemma}

\begin{proof}
If $k$ is strictly larger than the degree of $p$, then 
$\langle T_{\overline{h}}p,z^k\rangle_{H^2}=\langle p,z^kh\rangle_{H^2}=0$.
\end{proof}

\begin{lemma}\label{L:harmonic}
Let $h\in H^\infty$ with $\|h\|_{H^\infty}\le1$. Then,
for all $g\in H^\infty$, we have
\[
\|T_{\overline{h}}g-g\|_{H^2}^2\le 2(1-\Re h(0))\|g\|_\infty^2.
\]
\end{lemma}

\begin{proof}
Expanding the left-hand side, we obtain
\begin{align*}
\|T_{\overline{h}}g-g\|_{H^2}^2
&=\|T_{\overline{h}}g\|_{H^2}^2+\|g\|_{H^2}^2-2\Re\langle T_{\overline{h}}g,g\rangle_{H^2}\\
&\le 2\|g\|_{H^2}^2-2\Re\langle g,hg\rangle_{H^2}\\
&=2\int_0^{2\pi}|g(e^{i\theta})|^2\bigl(1-\Re h(e^{i\theta})\bigr)\,\frac{d\theta}{2\pi}\\
&\le2\|g\|_{H^\infty}^2\int_0^{2\pi}\bigl(1-\Re h(e^{i\theta})\bigr)\,\frac{d\theta}{2\pi}\\
&=2\|g\|_{H^\infty}^2(1-\Re h(0)).\qedhere
\end{align*}
\end{proof}

\begin{lemma}\label{L:H2Hb}
Let $b$ be non-extreme and let $(b,a)$ be a pair. 
If $h\in aH^\infty$, then
$T_{\overline{h}}$ is a bounded operator from $H^2$ into $\cH(b)$ and $\|T_{\overline{h}}\|_{H^2\to\cH(b)}\le\|h/a\|_{H^\infty}$.
\end{lemma}

\begin{proof}
Let $h=ah_0$, where $h_0\in H^\infty$. For $f\in H^2$, we have
\[
T_{\overline{b}}T_{\overline{h}}f
=T_{\overline{b}}T_{\overline{a}}T_{\overline{h}_0}f
=T_{\overline{a}}T_{\overline{h}_0}T_{\overline{b}}f
\]
so, by Theorem~\ref{T:f+}, $T_{\overline{h}}f\in\cH(b)$ and
$(T_{\overline{h}}f)^+=T_{\overline{h}_0}T_{\overline{b}}f$.
Consequently
\begin{align*}
\|T_{\overline{h}}f\|_{\cH(b)}^2
&=\|T_{\overline{h}}f\|_{H^2}^2
+\|T_{\overline{h}_0}T_{\overline{b}}f\|_{H^2}^2\\
&=\|T_{\overline{h}_0}T_{\overline{a}}f\|_{H^2}^2
+\|T_{\overline{h}_0}T_{\overline{b}}f\|_{H^2}^2\\
&\le\|h_0\|_{H^\infty}^2(\|T_{\overline{a}}f\|_{H^2}^2
+\|T_{\overline{b}}f\|_{H^2}^2)\\
&\le\|h_0\|_{H^\infty}^2\|f\|_{H^2}^2,
\end{align*}
the last inequality coming from the fact that $|a|^2+|b|^2=1$
a.e.\ on $\TT$, since $(b,a)$ is a pair.
\end{proof}

We now put these results together to produce a new, constructive proof of the following result.

\begin{theorem}\label{T:polysdense}
If $b$ is non-extreme, then polynomials are dense in $\cH(b)$.
\end{theorem}

\begin{proof}
Let $f\in\cH(b)$ and let $\epsilon>0$.
Pick $g_1,g_2\in H^\infty$ such that
$\|f-g_1\|_{H^2}\le\epsilon$ 
and $\|f^+-g_2\|_{H^2}\le\epsilon$.
Then pick $h\in aH^\infty$ such that $\|h\|_{H^\infty}\le1$ and 
$2(1-\Re h(0))\le\epsilon^2/(\|g_1\|_{H^\infty}^2+\|g_2\|_{H^\infty}^2)$. 
(For example, let $h$ be the outer function satisfying $h(0)>0$ and $|h|=\min\{1,n|a|\}$, where $n$ is chosen sufficiently large.) 
Finally, pick a polynomial $p$ such that 
$\|f-p\|_{H^2}\le\epsilon/\|h/a\|_{H^\infty}$. Then, by Lemma~\ref{L:Toeplitzpoly} $T_{\overline{h}}p$ is a polynomial, and we shall show that $\|f-T_{\overline{h}}p\|_{\cH(b)}\le 6\epsilon$.

First of all, using Lemma~\ref{L:harmonic} we have
\begin{align*}
\|f-T_{\overline{h}}f\|_{H^2}
&\le \|f-g_1\|_{H^2}+\|g_1-T_{\overline{h}}g_1\|_{H^2}+\|T_{\overline{h}}g_1-T_{\overline{h}}f\|_{H^2}\\
&\le \|f-g_1\|_{H^2}+\sqrt{2}(1-\Re h(0))^{1/2}\|g_1\|_{H^\infty}+\|g_1-f\|_{H^2}\\
&\le 3\epsilon.
\intertext{Likewise,}
\|f^+-T_{\overline{h}}f^+\|_{H^2}
&\le \|f^+-g_2\|_{H^2}+\|g_2-T_{\overline{h}}g_2\|_{H^2}+\|T_{\overline{h}}g_2-T_{\overline{h}}f^+\|_{H^2}\\
&\le \|f^+-g_2\|_{H^2}+\sqrt{2}(1-\Re h(0))^{1/2}\|g_2\|_{H^\infty}+\|g_2-f^+\|_{H^2}\\
&\le 3\epsilon.
\end{align*}
Now 
$T_{\overline{b}}(T_{\overline{h}}f)
=T_{\overline{h}}(T_{\overline{b}}f)
=T_{\overline{h}}(T_{\overline{a}}(f^+))
=T_{\overline{a}}(T_{\overline{h}}(f^+)),
$
so by Theorem~\ref{T:f+} we have 
$T_{\overline{h}}f\in\cH(b)$ and
$(T_{\overline{h}}f)^+=T_{\overline{h}}(f^+)$,
and
\[
\|f-T_{\overline{h}}f\|_{\cH(b)}^2
=\|f-T_{\overline{h}}f\|_{H^2}^2
+\|f^+-T_{\overline{h}}f^+\|_{H^2}^2
\le 9\epsilon^2+9\epsilon^2\le (5\epsilon)^2.
\]
Finally, using Lemma~\ref{L:H2Hb}, we have
\begin{align*}
\|f-T_{\overline{h}}p\|_{\cH(b)}
&\le \|f-T_{\overline{h}}f\|_{\cH(b)}+\|T_{\overline{h}}f-T_{\overline{h}}p\|_{\cH(b)}\\
& \le \|f-T_{\overline{h}}f\|_{\cH(b)}+\|h/a\|_{H^\infty}\|f-p\|_{H^2}\\
&\le 5\epsilon+\epsilon=6\epsilon,
\end{align*}
as claimed. 
\end{proof}

\section{Toeplitz approximation in $\cH(b)$}\label{S:Toeplitz}

Recapitulating the proof of Theorem~\ref{T:polysdense},
given $f\in\cH(b)$, we can approximate it in $\cH(b)$ by a polynomial of the form $T_{\overline{h}}p$, where $h\in H^\infty$ and $p$ is a polynomial. Indeed, by the triangle inequality,
\[
\|f-T_{\overline{h}}p\|_{\cH(b)}
\le \|f-T_{\overline{h}}f\|_{\cH(b)}
+\|T_{\overline{h}}f-T_{\overline{h}}p\|_{\cH(b)}.
\]
The first term on the right-hand can be made small by choosing $h$ with $\|h\|_{H^\infty}\le1$ and $h(0)$ sufficiently close to $1$. If, in addition, $h\in aH^\infty$, then the second term can be made small by choosing $p$ sufficiently close to $f$ in $H^2$, for example a Taylor partial sum of $f$.

This construction presupposes that $b$ is non-extreme. Indeed it must, since otherwise $\cH(b)$ may well contain no non-zero polynomials. However, the fact that $\|f-T_{\overline h}f\|_{\cH(b)}$ can be made small remains true even in the case when $b$ is extreme. We isolate the idea in the following approximation theorem, valid for all $b$, extreme and non-extreme.

\begin{theorem}\label{T:Toeplitzapprox}
Let $(h_n)_{n\ge1}$ be a sequence in $H^\infty$ such that
$\|h_n\|_{H^\infty}\le1$  and  $\lim_{n\to\infty}h_n(0)=1$.
Then, given $b$ in the unit ball of $H^\infty$ and $f\in\cH(b)$, 
we have $T_{\overline{h}_n}f\in\cH(b)$ for all $n$ and
\[
\lim_{n\to\infty}\|T_{\overline{h}_n}f-f\|_{\cH(b)}=0.
\]
\end{theorem}

The proof of this result requires a little more background on de Branges--Rovnyak spaces, which we now briefly summarize.

Let $b\in H^\infty$ with $\|b\|_{H^\infty}\le1$. 
We define the space $\cH(\overline{b})$ in the same way as $\cH(b)$,
but with the roles of $b$ and $\overline{b}$ interchanged. Thus $\cH(\overline{b})$
is the image of $H^2$ under the operator $(I-T_{\overline b}T_b)^{1/2}$,
with norm defined  by
\[
\|(I-T_{\overline{b}}T_b)^{1/2}f\|_{\cH(\overline{b})}:=\|f\|_{H^2}
\qquad(f\in H^2\ominus \ker(I-T_{\overline{b}}T_b)^{1/2})).
\]

The spaces $\cH(b)$ and $\cH(\overline{b})$ are related through the following theorem.

\begin{theorem}[\protect{\cite[\S II-4]{Sa94}}]\label{T:hbbar}
Let $b$ be an element of the unit ball of $H^\infty$ and let $f\in H^2$. 
Then $f\in\cH(b)$ if and only $T_{\overline{b}}f\in\cH(\overline{b})$, and in this case
\[
\|f\|_{\cH(b)}^2=\|f\|_{H^2}^2+\|T_{\overline{b}}f\|_{\cH(\overline{b})}^2.
\]
\end{theorem}

The advantage of $\cH(\overline{b})$ over $\cH(b)$, for our purposes at least,
is that it has another description making it a little more amenable.

\begin{theorem}[\protect{\cite[\S III-2]{Sa94}}]\label{T:Hbbar}
Let $b\in H^\infty$ with $\|b\|_{H^\infty}\le1$.
Let  $\rho:=1-|b|^2$ on $\TT$,
let $H^2(\rho)$ be the closure of the polynomials in $L^2(\TT,\,\rho\,d\theta/2\pi)$ and let $J_\rho:H^2\to H^2(\rho)$ be the natural inclusion. Then $J_\rho^*$ is an isometry of $H^2(\rho)$ onto
$\cH(\overline{b})$.
\end{theorem}

Using this result, we can prove a version of Theorem~\ref{T:Toeplitzapprox} for $\cH(\overline{b})$.

\begin{theorem}\label{T:Toeplitzbbar}
Let $(h_n)_{n\ge1}$ be a sequence in $H^\infty$ such that
$\|h_n\|_{H^\infty}\le1$  and  $\lim_{n\to\infty}h_n(0)=1$.
Then, given $b$ in the unit ball of $H^\infty$ and $f\in\cH(\overline{b})$, 
we have $T_{\overline{h}_n}f\in\cH(\overline{b})$ for all $n$ and
\[
\lim_{n\to\infty}\|T_{\overline{h}_n}f-f\|_{\cH(\overline{b})}=0.
\]
\end{theorem}

\begin{proof}
Let $\rho:=1-|b|^2$ on $\TT$ and define $H^2(\rho)$ and $J_\rho$ as in the preceding theorem. Given $h\in H^\infty$, let $M_h:H^2(\rho)\to H^2(\rho)$ be the operator of multiplication by $h$, namely $M_hg:=hg~(g\in H^2(\rho))$. Note that $M_hJ_\rho=J_\rho T_h$, so, taking adjoints,
we have
\begin{equation}\label{E:intertwine}
J_\rho^*M_h^*=T_{\overline{h}}J_\rho^*.
\end{equation}

Now let $f\in \cH(\overline{b})$. 
By Theorem~\ref{T:Hbbar}, 
there exists $g\in H^2(\rho)$ such that $f=J_\rho^*g$. 
Using \eqref{E:intertwine}, we have
$T_{\overline{h}_n}f=T_{\overline{h}_n}J_\rho^*g=J_\rho^*M_{h_n}^*g$. 
As $J_\rho^*$ is an isometry of $H^2(\rho)$ onto $\cH(\overline{b})$, it follows that 
$T_{\overline{h}_n}f\in\cH(\overline{b})$ and that
\[
\|T_{\overline{h}_n}f-f\|_{\cH(\overline{b})}
=\|M_{h_n}^* g-g\|_{H^2(\rho)}.
\]

It therefore remains to show that $\lim_{n\to\infty}\|M_{h_n}^* g-g\|_{H^2(\rho)}=0$ for all $g\in H^2(\rho)$. It suffices to prove this when $g\in H^\infty$, because $H^\infty$ is dense in $H^2(\rho)$ and the operators $M_{h_n}^*$ are uniformly bounded in norm (by $1$). For $g\in H^\infty$, the same proof as that of Lemma~\ref{L:harmonic} gives that
\[
\|M_{h_n}^* g-g\|_{H^2(\rho)}^2\le 2(1-\Re h_n(0))\|g\|_{H^\infty}^2,
\]
and as the right-hand side tends to zero, the proof is complete.
\end{proof}

Finally, we deduce the corresponding result for $\cH(b)$.

\begin{proof}[Proof of Theorem~\ref{T:Toeplitzapprox}]
Let $f\in\cH(b)$. By Theorem~\ref{T:hbbar} 
we have $T_{\overline{b}}f\in\cH(\overline{b})$.
For each $n$,
we have $T_{\overline{b}}T_{\overline{h}_n}f=
T_{\overline{h}_n}T_{\overline{b}}f\in\cH(\overline{b})$,
and hence $T_{\overline{h}_n}f\in\cH(b)$.
Also, by Theorem~\ref{T:hbbar} again,
\[
\|T_{\overline{h}_n}f-f\|_{\cH(b)}^2
=\|T_{\overline{h}_n}f-f\|_{H^2}^2+\|T_{\overline{h}_n}T_{\overline{b}}f-T_{\overline{b}}f\|_{\cH(\overline{b})}^2.
\]
The second term on the right-hand side tends to zero
by Theorem~\ref{T:Toeplitzbbar}. The first term tends
to zero by Lemma~\ref{L:harmonic} and the density of $H^\infty$ in $H^2$.
\end{proof}

\section*{Acknowledgement}
Part of this research was carried out during a Research-in-Teams meeting at the Banff International Research Station (BIRS). We thank BIRS for its hospitality.

\end{document}